\documentclass{amsart}
\usepackage[nobysame]{amsrefs}
\usepackage{amsmath}
\usepackage{amssymb}
 \usepackage{graphicx}

\theoremstyle{plain}
\newtheorem{theorem}{Theorem}[section]
\newtheorem{lemma}{Lemma}[section]
\newtheorem{corollary}{Corollary}[section]
\newtheorem{proposition}{Proposition}[section]

\newtheorem*{claim*}{Claim}
\newtheorem*{lemma*}{Lemma}
\newtheorem*{theorem*}{Theorem}

\theoremstyle{definition}
\newtheorem{definition}{Definition}[section]

\theoremstyle{remark}

\newtheorem*{remark*}{Remark}

\begin{document}

\title{Finite type domains with hyperbolic orbit accumulation points}

\author{Bingyuan Liu}
\address{Department of Mathematics, Washington University, Saint Louis, USA}
\email{bingyuan@math.wustl.edu}



\date{\today}



\begin{abstract}
In this paper, finite type domains with hyperbolic orbit accumulation points are studied. We prove, in case of $\mathbb{C}^2$, it has to be a (global) pseudoconvex domain, after an assumption of boundary regularity. Moreover, one of the applications will realize the classification of domains within this class, precisely the domain is biholomorphic to one of the ellipsoids $\lbrace (z, w): \vert z\vert^{2m}+\vert w\vert^2<1, m\in\mathbb{Z}^+\rbrace$. This application generalizes \cite{BP98} in which the boundary is assumed to be real analytic for the case of hyperbolic orbit accumulation points. 
\end{abstract}

\maketitle

\section{Introduction}

Let $\Omega$ be a smooth bounded domain in $\mathbb{C}^n$ and $p\in\partial\Omega$. It was a long time since Greene-Krantz posted their conjecture in \cite{GK91}, which states, if $p$ is a boundary orbit accumulation point, then $p$ is a point of finite type. By orbit accumulationn boundary point $p$, we mean a boundary point $p\in\partial\Omega$ such that $\lim_{j\to\infty}f_j(q)=p$ where $q\in\Omega$ and $f_j\in\rm Aut(\Omega)$. There are numerous works on this problem for 20 years by many mathematicians, e.g., we just mention some (in alphabet order), Eric Bedford, Jisoo Byun, Robert Greene, Kang-Tae Kim, Sung-Yeon Kim, Mario Landucci, Steven Krantz, Sergey Pinchuk, Jean-Christophe Yoccoz. Partial results have already been achieved, e.g. \cite{BP88}, \cite{BP91}, \cite{BP98}, \cite{GK91}, \cite{GK93}, \cite{KY11}, \cite{Ki12}, \cite{Kr11}, \cite{Kr12}, \cite{La04}. Among those, recently, Sung-Yeon Kim publishes the result in her paper \cite{Ki12} which proves the Greene-Krantz conjecture in case of hyperbolic orbit accumulation points. In this note, we consider the domain with noncompact automorphism groups from another point of view, namely, to check whether it is globally pseudoconvex. By pseudoconvex, we usually mean here weakly pseudoconvex, since a strongly pseudoconvex domain with noncompact automorphism groups will make the domain a ball by the well-known Wong-Rosay theorem (see \cite{Ro79} and \cite{Wo77}). 

Let $\Omega\in\mathbb{C}^2$ be a domain with real analytic boundary. It was shown by Bedford-Pinchuk that noncompact automorphism group implies $\Omega$ is biholomorphic to one of the ellipsoids $\lbrace (z,w): \vert z\vert^{2m}+\vert w\vert^2, m\in\mathbb{Z}^+ \rbrace$. On can easily check that ellipsoids are globally pseudoconvex. However, if the problem passes to the category of smooth boundary, i.e. the defining function is $C^\infty$, the answer is not so clear as the domain with real analytic boundary. The difficulty is that some of the tools for real analytic boundary like Segre variety and analytic variety, cannot be used. Shortly after \cite{BP88}, Catlin pointed out (unpublished) a pseudoconvex domain with boundary of finite type with noncompact automorphism group should be enough to be an ellipsoid (analytic is not necessary). However, one still wonders if ``pseudoconvex'' can be removed. The author will try to replace ``pseudoconvex'' with other assumptions, although he is unable to remove it completely so far.

In the present note, we mainly work on the following result.

\begin{theorem}\label{mainth}
Let $\Omega\subset\mathbb{C}^2$ be a bounded domain with smooth boundary of finite type. Suppose that the Bergman kernel of $\Omega$ extends to $\overline{\Omega}\times\overline{\Omega}$ minus the boundary diagonal set as a locally bounded function. Let $p\in\partial\Omega$ be a hyperbolic orbit accumulation point. Then $\Omega$ is globally pseudoconvex. 
\end{theorem}

For the sake of completeness, we define the so-called ``orbit accumulation points''.

\begin{definition}
Let $\Omega$ be a smoothly bounded domain in $\mathbb{C}^2$. If there exist points $q\in\Omega$, $p\in\partial\Omega$ and a sequence $\lbrace f_\nu\rbrace\subset\rm Aut(\Omega)$ such that $f_\nu(q)$ converges to $p$. The point $p$ is called an orbit accumulation point. If $f_\nu^{-1}(q)$ converges to another boundary point $\tilde{p}\in\partial\Omega-p$, where $f_\nu^{-1}$ is the inverse of $f_\nu$, then $p$ is called a hyperbolic orbit accumulation point.
\end{definition}

The method of proof involves analysis of $\partial\Omega$ and the tools borrowed from CR geometry. We also try to write this note as concise as possible.

We should remark that in Theorem \ref{mainth}, ``finite type'' can be replaced with ``boundary satisfying condition R in sense of Bell with $p$ that is holomorphically simple (i.e. there is no complex variety through $p$ that lies in the boundary)''. Furthermore, the result is extended to higher dimensions in the author's forthcoming paper \cite{Li13}.

We also remark that for general case (not the hyperbolic accumulation pints), the method might not work. It is because that the boundary might not be defined by a rigid equation then, even locally. Moreover, the condition of extension of the Bergman kernel to the boundary minus the diagonal set is verifiable when the $\overline{\partial}$-Neumann problem is pseudolocal.

\section{Preliminary}

The Hilbert transform has a long history in both fields of one complex variable and several complex variables. 

In particular, the Hilbert transform on the unit disc is most important. Let $u$ be a real-valued function on $\partial\Delta$. Setting $z=re^{i\theta}$ with $0\leq r<1$ and $\zeta=e^{it}$, we define
\begin{equation*}
T'u(re^{i\theta})=\frac{1}{2\pi}\int_{-\pi}^{\pi} K_r(t)u(e^{i(\theta-t)})\,\mathrm{d}t,
\end{equation*}
where
\begin{equation*}
K_r(t)=\frac{2r\sin t}{1-2r\cos t+r^2}
\end{equation*}
is the Hilbert kernel (closely related to the well-known Poisson kernel).

Roughly speaking, the Hilbert transform is the limit function $T'u(re^{i\theta})$ as $r\rightarrow 1^-$. One can treat the following fact as the definition of the Hilbert transform.

For $u\in C^\alpha(\partial\Delta)$,
\begin{equation*}
Tu(e^{i\theta})=\mathrm{p.v.}\frac{1}{2\pi}\int_{-\pi}^{\pi} \frac{u(e^{i(\theta-t)})}{\tan(t/2)}\,\mathrm{d}t
\end{equation*}

Due to the well-known Riemann mapping theorem and Carath\'{e}odory theorem, the Hilbert transform works on an arbitrary non-empty simply connected open subset of the complex number plane $\mathbb{C}$ which is not all of $\mathbb{C}$, whose boundary is a Jordan curve.

The following fact from \cite{MP06} gives readers a nice intuition.

\begin{theorem}[See Lemma 2.25 of Chapter IV in \cite{MP06}]
The Hilbert transform $Tu$ on $\partial\Delta$ is the boundary value on $\partial\Delta$ of the unique harmonic conjugate in $Delta$ of the harmonic Poisson extension $Pu$, that vanishes at $0\in\Delta$.
\end{theorem}

In this note, we mainly use the modified Hilbert transform as following.

\begin{definition}
The modified Hilbert transform $T_1$ is defined by
\begin{equation*}
T_1u(e^{i\theta}):=Tu(e^{i\theta})-Tu(1),
\end{equation*}
where $u\in C^\alpha(\partial\Delta)$ and $T$ is the (classical) Hilbert transform.
\end{definition}

The modified Hilbert is a mild modification such that it suits more applications. Specifically, it is used to solve the Bishop equation, namely,
\begin{equation*}
U(\zeta)=-T_1(\rho(Z(\cdot),\overline{Z(\cdot)}, U(\cdot)))(\zeta)+c, \zeta\in S^1
\end{equation*}
where $T_1$ is modified Hilbert transform, $\rho$ is a defining function, $Z$ is a given known analytic disc and $c$ is a given real number. 

In fact, Let $M$ be the hypersurface in $\mathbb{C}^n=(z, w)$ with a defining function $\rho$, where $z\in\mathbb{C}^{n-1}$ and $z\in\mathbb{C}$. Suppose we are given an arbitrary analytic disc $Z: \overline{\Delta}\rightarrow\mathbb{C}^{n-1}$, where $\overline{\Delta}$ denotes the unit close disc in $\mathbb{C}$, and a real number $c$. By solving the Bishop equation, one can obtain an analytic disc $A(\zeta)=(Z(\zeta), W(\zeta))$ attached to $M$. Moreover the projection of $A(\zeta)$ onto $z$-component is exactly $Z(\zeta)$ and $\Re w(1)=c$. Precisely and more generally, by the solution of the Bishop equation, we usually mean the following (see also \cite{BE99}, \cite{Bo91} and \cite{MP06}):

Let $M$ be a smooth generic submanifold through $0$, given by $v=\rho(z, \bar{z}, u), z\in\mathbb{C}^{n-d}, w=u+iv\in\mathbb{C}^d$ of class $C^2$ with $h(0)=0$, $Dh(0)=0$ and let $0<\alpha<1$. Then there exists $\epsilon>0$ such that for any analytic disc $Z: \Delta\rightarrow\mathbb{C}^n$ in $C^{0, \alpha}(\Delta)$ with $\parallel Z\parallel_{0, \alpha}<\epsilon$ and for any $c\in\mathbb{R}^d$ with $\vert c\vert<\epsilon$, there exists a unique (small) analytic disc $A(\zeta)=(Z(\zeta), W(\zeta))$ of class $C^{0, \alpha}$ attached to $M$ such that $\Re W(1)=c$. In addition, if $\rho$ is of class $C^k$ $(k\geq 0)$, then there is an $\epsilon>0$ such that $H: C^{k, \alpha(S^1}, \mathbb{C}^{n-d})\times\mathbb{R}^d\rightarrow C^{k, \alpha}(S^1, \mathbb{R}^d)$ depends in a $C^{k-1}$ fashion on $c\in\mathbb{R}^d$ and $Z\in C^{k, \alpha}(S^1, \mathbb{C}^{n-d})$, with $\vert c\vert<\epsilon$ and $\parallel Z \parallel_{k, \alpha}<\epsilon$. Here $C^{k, \alpha}$ denotes the class of H\"{o}lder space and the norm $\displaystyle\parallel Z \parallel_{k, \alpha}=\sum\limits_{\mid\gamma\mid\leq k}\parallel D^\gamma Z\parallel_{C(S^1)}+\sum\limits_{\mid\gamma\mid=k}\parallel D^\gamma Z\parallel_{C^{0, \alpha}(S^1)}$.

With the solution of the Bishop equation, Tr\'{e}preau and Tumanov proved two celebrated theorems (see \cite{Tr86} for the case of hypersurfaces and \cite{Tu90} for the case of submanifolds). Specifically, for the case of hypersurfaces, Tr\'{e}preau proved that each CR function on a minimal hypersurface, can be extended to a holomorphic function on one side of the hypersurface locally. 

Another important tool is developed by Kim(Kang-Tae)-Yoccoz in \cite{KY11}. They borrowed well-known results from dynamical systems in \cite{PY90} to study contractions in CR geometry. Indeed, in the case of hypersurfaces, they are able to prove that a germ of a hypersurface $M$ that admits contractions has a defining function that can be written as $v=P(z, \bar{z})$, where $P$ is a weighted homogeneous polynomial. More recently, Kim (Sung-Yeon) proved a theorem solving the Greene-Krantz conjecture in case of hyperbolic orbit accumulation points in \cite{Ki12}. She proved also in case of hyperbolic orbit accumulation points, around the orbit accumulation point, the germ of hyperbolic orbit accumulation point admits a contraction. This finishes partially the 20-years old conjecture. 

In this note, we assume the Bergman kernel of $\Omega$ extends to $\overline{\Omega}\times\overline{\Omega}$ minus the boundary diagonal set as a locally bounded function. We should discuss its existence in the nonpseudoconvex case, and otherwise, there is no sense to prove the result. By \cite{Kr11}, the nonpseudoconvex domain with a desired Bergman kernel exists, e.g. the shell.  

\section{Proof of Theorem \ref{mainth}}

To prove Theorem \ref{mainth}, we need several lemmas.

\begin{lemma}\label{thefirstlemma}
Let $\Omega$ be a smooth bounded domain with boundary of finite type in $\mathbb{C}^2$. Suppose that the Bergman kernel of $\Omega$ extends to $\overline{\Omega}\times\overline{\Omega}$ minus the boundary diagonal set as a locally bounded function. Then for any hyperbolic orbit accumulation boundary point $p$, there exists a contraction $f\in \rm Aut(\Omega)\cap\rm Diff(\Omega)$ at $p$. Moreover, $\Omega$ is biholomorphic to the domain $\lbrace (z, w): v<P(z, \bar{z})\rbrace$, where $P$ is a homogeneous polynomial. This biholomorphism extends smoothly up to boundary around $p$.
\end{lemma}

\begin{proof}
In view of Kim's theorem in \cite{Ki12}, the ``pseudoconvex'' can be replaced with ``finite type'' to satisfy ``condition R'' in sense of Bell. By Kim-Yocooz's theorem in \cite{KY11}, one obtains the desired result easily, because a weighted homogeneous polynomial in one complex variable is homogeneous.
\end{proof}

Lemma \ref{thefirstlemma} gives a good starting point for our work. In fact, it shows hyperbolic orbit accumulation points make $\partial\Omega$ defined with a rigid equation. By the rigid equation, we mean a equation without involving the real argument of $w$. 

Since we are planning to work with a hypersurface, not a domain, we introduce a definition to name the pseudoconvexity of hypersurfaces. We want to remark our definition is slightly different from that in standard CR geometry in order to be adapted in our discussion. We also want to distinguish the two sides of a germ of hypersurface by the pseudoconvex side and the pseudoconcave side (see Section \ref{proofofproposition}). 

\begin{definition}\label{pseudoconvexity}
Let $M\subset\mathbb{C}^2$ be a piece of hypersurface defined by $v=\rho(z, \bar{z}, u)$. If there is a pseudoconvex (respectively, strongly pseudoconvex) domain $G\subset\mathbb{C}^2$ so that $M\subset\partial G$ and $G$ is contained in $\lbrace (z,\bar{z}, u,v)\in\mathbb{C}^2: v<\rho(z,\bar{z},u)\rbrace$, then $S$ is said to be pseudoconvex (respectively, strongly pseudoconvex). If $G$ is contained in $\lbrace (z,\bar{z}, u,v)\in\mathbb{C}^2: v>\rho(z,\bar{z},u)\rbrace$, $M$ is said to be pseudoconcave (respectively, strongly pseudoconcave). 
\end{definition}

With the definition above, the pseudoconvexity in a hypersurface is defined by a rigid homogeneous polynomial will appear a nice property as what we are going to show.

\begin{lemma}\label{lemmaabouttheta}
Let $M$ be a smooth hypersurface passing through $(0, 0)$ in $\mathbb{C}^2$ defined by $0=\rho=v-P(z, \bar{z})$, where $P$ is a homogeneous polynomial with respect to $z$. Then locally around $(0, 0)$, the pseudoconvexity is solely determined by the argument of $z$. 
\end{lemma}

\begin{proof}
By elementary calculations, one observes that the sign of Levi form is same as the one of $\Delta P$. Indeed, given defining function $\rho$, in $\mathbb{C}^2$ the Levi form is defined by $\dfrac{\partial^2\rho}{\partial z\partial \bar{z}}\left\vert\dfrac{\partial\rho}{\partial w}\right\vert^2-2\Re\left\lbrace\dfrac{\partial^2\rho}{\partial z\partial \bar{w}}\left(\dfrac{\partial\rho}{\partial\bar{z}}\right)\left(\dfrac{\partial\rho}{\partial w}\right)\right\rbrace+\dfrac{\partial^2\rho}{\partial w\partial \bar{w}}\left\vert\dfrac{\partial\rho}{\partial z}\right\vert^2$. This is because of the definition of Levi form with putting $\dfrac{\partial\rho}{\partial w}\dfrac{\partial}{\partial z}-\dfrac{\partial\rho}{\partial z}\dfrac{\partial}{\partial w}$ and its conjugate as the basis of complex tangent space of $M$. But, thanks to the explicit formula $\rho=v-P(z, \bar{z})$, the Levi form is $\dfrac{\partial^2 P}{\partial z\partial \bar{z}}\left\vert\dfrac{\partial\rho}{\partial w}\right\vert^2$ and clearly its sign is determined by $\Delta P$. Recall also, that in $\mathbb{C}^2$ we call a point $(z, w)$ Levi-flat, if its Levi form vanishes at that point. 

With the computation above, one can verify the following statement.
\begin{enumerate}
\item The point $(z, w)\in M$ is Levi-flat if and only if $\Delta P=0$ at $(z, w)$,
\item the point $(z, w)\in M$ is pseudoconvex in sense of Definition \ref{pseudoconvexity} if and only if $\Delta P\leq 0$ at $(z, w)$,
\item the point $(z, w)\in M$ is pseudoconcave in sense of Definition \ref{pseudoconvexity} if and only if $\Delta P\geq 0$ at $(z, w)$.
\end{enumerate}
So we just need to investigate the Laplacian of $P$. For this, we write $\Delta=\dfrac{1}{r}\dfrac{\partial}{\partial r}+\dfrac{\partial^2}{\partial r^2}+\dfrac{1}{r^2}\dfrac{\partial^2}{\partial\theta^2}$. Since $P$ is a homogeneous polynomial, we find $\Delta P=0$ just involves the argument $\theta$ and not $r$. It is clear that the solution gives finite numbers for $\theta$ which is because of analyticity of $\Delta P$. Moreover the solution of $\Delta P=0$ is union of intervals of $\theta$, e.g. $\lbrace \theta\in(\theta_1, \theta_2)\cup(\theta_3, \theta_4)\cup(\theta_7, \theta_9)\rbrace$. It completes the proof.
\end{proof}

Moreover, let $P_z: \mathbb{C}^2\rightarrow\mathbb{C}$ be the projection which maps point $(z, w)$ to the first component $z$. Then the pseudoconvex piece of $M$ is given by $P_z^{-1}(\cup_i S_i)\cap M$ where $S_i$ denotes the sectors where $\Delta\rho\leq 0$. The following Proposition gives a nice property which is useful in the proof of Theorem \ref{mainth}. Also, it is independently interesting, so we prove it in Section \ref{proofofproposition}.

\begin{proposition}\label{prop}
Let $D(0, \epsilon)$ be a disc center at 0 with radius $\epsilon$. Let $M\subset\mathbb{C}^2$ be a hypersurface passing through $(0,0)$ and satisfying the following properties: for $\lbrace\theta\rbrace_{i=1}^n\in[0, 2\pi)$ so that
\begin{enumerate}
\item $(z=x+iy=re^{i\theta}, w)\in M\cap P_z^{-1}D(0,\epsilon)$ is Levi-flat if and only if $\theta=\theta_i, i=1,2,\dots, n$. We denote $\Sigma_0$ as the set $\lbrace(z=re^{i\theta}):\theta=\theta_i, i=1,2,\dots, n\rbrace$. 
\item $(z=x+iy=re^{i\theta}, w)\in M\cap P_z^{-1}D(0,\epsilon)$ is strongly pseudoconvex if and only if there is $k\in\lbrace 1,2,\dots,n-1\rbrace$ so that $\theta_k<\theta<\theta_{k+1}$. We denote $\Sigma_+$ as the set $\lbrace(z=re^{i\theta}):\theta_k<\theta<\theta_{k+1}\rbrace$. 
\item $(z=x+iy=re^{i\theta}, w)\in M\cap P_z^{-1}D(0,\epsilon)$ is strongly pseudoconcave if and only if there is $l\in\lbrace 1,2,\dots,n-1\rbrace$ so that $\theta_l<\theta<\theta_{l+1}$. We denote $\Sigma_-$ as the set $\lbrace(z=re^{i\theta}):\theta_l<\theta<\theta_{l+1}\rbrace$. 
\end{enumerate}
Then it has a holomorphic extension through $(0,0)$ to both sides. In other words, there exists a polydisc $U_0$ centered at $(0, 0)$ and $M$ divides $U_0$ into two open subsets $U_0^+$ and $U_0^-$ respectively. For an arbitrary continuous CR function $f$ defined in $U_0\cap M$, there exist two holomorphic functions $\tilde{f}_1\in\mathcal{O}(U_0^+)$ and $\tilde{f}_2\in\mathcal{O}(U_0^-)$ such that $\tilde{f}_1\vert_{U_0\cap M}=\tilde{f}_2\vert_{U_0\cap M}=f$.
\end{proposition}

We are ready to prove our main theorem now.

\begin{proof}[Proof of Theorem \ref{mainth}]
The assumption of the existence of hyperbolic orbit accumulation point $p$, and we let $\tilde{p}$ be the dual hyperbolic orbit accumulation point of $p$. We also let $v=\rho(z, \bar{z})$ be the local defining function of $\partial\Omega$ with $p=(0, 0)$.

Let us assume firstly there exists a strongly pseudoconcave point in any neighborhood of $p=(0, 0)$. Because, otherwise there is no strongly pseudoconcave point in a neighborhood $U_0$ of $q$ in $\partial\Omega$; then $U_0$ is a pseudoconvex piece of $\partial\Omega$ including $p$. By the Proposition $1$ in \cite{Kr12}, there is no strongly pseudoconcave point (in a regular sense) on $\partial\Omega$ and the theorem follows.

We can now also assume there is a strongly pseudoconvex point in any neighborhood of $p=(0, 0)$. Because,  otherwise, there are only strongly and weakly pseudoconcave points around $p$. By looking at Lemma \ref{lemmaabouttheta} and an a special case of Proposition \ref{prop} (let $\Sigma_+$ vanish), one can see it is a holomorphic extension through $(0,0)$ to outside. But this contradicts with a similar argument as in Greene-Krantz \cite{GK93} as follow.

By Cartan's Theorem in \cite{Na95}, $J(\phi_j)$ must tend to zero at any point in $\Omega$, where $J$ denotes the Jacobian, because $\lbrace\phi_j\rbrace\rightarrow p$. On the other hand $\lbrace\phi_j^{-1}\rbrace$ is a family of automorphism and hence a normal family. Moreover, CR functions $J(\phi_j^{-1})\vert_{\partial\Omega}$ extends to the holomorphic hull $\widehat{\Omega}$ of $\Omega$ which at least contains $p$ as an interior point by the discussion above. Now extensions $\widehat{ J(\phi_j^{-1})}$ of $J(\psi_j^{-1})$ converges uniformly on any compact subset of $\widehat{\Omega}$. Pick up an compact subset $K$ of $\widehat{\Omega}$ containing $p$, and then the limit of subsequence of $\widehat{\vert J(\psi_j^{-1})\vert}$ is bounded because $\lbrace\vert J(\psi_j^{-1})\vert\rbrace_{j=1}^\infty$ are uniformly bounded. But this is impossible because $\vert J(\psi_j)\vert$ goes to zero.

Suppose now there exist a strongly pseudoconcave point and a storngly pseudoconvex point in any neighborhood of $p=(0, 0)$. By looking at Lemma \ref{lemmaabouttheta}, it forces both of inequalities $\Delta\rho<0$ and $\Delta\rho>0$ having solutions. By continuity, there must be (at least) two sectors $\lbrace (r, \theta): S_1:=(\theta_j, \theta_{j+1})\rbrace$ and $\lbrace (r, \theta): S_2:=(\theta_k, \theta_{k+1})\rbrace$ such that $P_z^{-1}S_1\cap\partial\Omega$ is strongly pseudoconvex but $P_z^{-1}S_2\cap\partial\omega$ is strongly pseudoconcave.

In view of Proposition \ref{prop}, we see that any CR function defined in a germ of $\Omega$ at $p$ will extend to a holomorphic function to each side of $\partial\Omega$ (the extension to outside is what we concern). We use the similar argument mentioned above to find the contradiction again. 
\end{proof}

\section{Proof of Proposition \ref{prop}}\label{proofofproposition}

We introduce the following definition to name the both sides of a hypersurface.

\begin{definition}
Let $M\subset\mathbb{C}^2$ be a pseudoconvex hypersurface in the sense of Definition \ref{pseudoconvexity}, then we call the side of $M$ contained in $\lbrace (z,\bar{z}, u,v)\in\mathbb{C}^2: v>\rho(z,\bar{z},u)\rbrace$ as the pseudoconvex side and otherwise, the pseudoconcave side. Let $M\subset\mathbb{C}^2$ be a pseudoconcave hypersurface, then we call the side of $M$ contained in $\lbrace (z,\bar{z}, u,v)\in\mathbb{C}^2: v>\rho(z,\bar{z},u)\rbrace$ as the pseudoconcave side and otherwise, the pseudoconvex side.
\end{definition}

To prove the holomorphic extension in Proposition \ref{prop}, we study the exit vectors of analytic discs. The exit vectors of an analytic disc is defined by:

\begin{equation*}
\left.-\frac{\partial A}{\partial\zeta}(\zeta)\right\vert_{\zeta=1}=\left.-\frac{\partial A}{\partial r}(r)\right\vert_{r=1}=\left.i\frac{\partial A}{\partial\theta}(e^{i\theta})\right\vert_{\theta=0},
\end{equation*}
where the first derivative is defined to be one-side derivative and the last two equality follows by chain rule.

Informally speaking, the direction which the exit vector points to is the direction of the holomorphic extension. Hence, we need to find at least two analytic disc whose exit vectors point to both sides. For the following paragraphs, we look for the analytic disc $A(z)=(Z(z)=X(z)+iY(z), W(z)=U(Z)+iV(z))$ whose exit vector point to the pseudoconvex side, i.e. $\dfrac{\partial U}{\partial\theta}(0)>0$.

It is easy to show with Riemann mapping theorem and Carath\'{e}dory theorem that there exist analytic discs $\lbrace Z'_n\rbrace$ in $\mathbb{C}$ satisfying:
\begin{enumerate}
\item The images $D'_n$ of $\lbrace Z'_n\rbrace$ are totally contained in $\Sigma_+\cup(0,0)$, so that  $Z'_n(-1)=q$ for all $n$ where $q$ is a strongly pseudoconvex point;
\item $\arg Z'_n(-1)=\arg Z'_n(1)=\dfrac{\theta_{k+1}-\theta_k}{2}$;
\item $Z'_n(1)\rightarrow (0, 0)$ as $n\rightarrow\infty$ and the boundaries of images are smooth; \item $D'_n$ are symmetric with respect to the ray $\theta=\dfrac{\theta_{k+1}-\theta_k}{2}$.
\end{enumerate}

By the solution of the Bishop equation, there exist a family of analytic discs $A'_n: \overline{\Delta}\rightarrow\mathbb{C}^2$ attached to $P_z^{-1}\Sigma_+\cap M$. Let $A'_n(z)=(Z'_n(z)=X'_n(z)+iY'_n(z), W'_n(z)=U'_n(Z)+iV'_n(z))$.

The first observation on exit vectors of $\lbrace Z'_n\rbrace$ is, $\dfrac{\partial U'_n}{\partial\theta}(0)\geq 0$. Otherwise, $\dfrac{\partial U'_{n_0}}{\partial\theta}(0)<0$ for some ${n_0}$, and by a well-known argument of the translation of a nontangent analytic disc  (see 2.12 of Chapter V in \cite{MP06}), there is $V_0\subset M\cap P_z^{-1}\Sigma_+$ containing $A'_{n_0}(1)$ such that each of continuous CR function in $V_0$ can extends to the pseudoconcave side which contradicting the Lewy extension theorem ($V_0$ is totally contained in a strongly pseudoconvex piece of $M$). 

Indeed, we will prove that, there exists a family of analytic disc $Z_n$ satisfying not only the properties above, but also a better uniform lower bound of $\dfrac{\partial U_n}{\partial\theta}(0)$. 

\begin{lemma}
There exist analytic discs $\lbrace Z_n\rbrace$ in $\mathbb{C}$ of which the images are totally contained in $\Sigma_+\cup(0,0)$, so that  $Z_n(-1)=q$ for all $n$ where $q$ is a strongly pseudoconvex point and $\arg Z_n(-1)=\arg Z_n(1)=\dfrac{\theta_{k+1}-\theta_k}{2}$; of which $Z_n(1)\rightarrow (0, 0)$ as $n\rightarrow\infty$ and the boundaries of images are smooth, and the images are symmetric with respect to the ray $\theta=\dfrac{\theta_{k+1}-\theta_k}{2}$. It also satisfies the following property, there exist $\epsilon_0>0$ so that $\dfrac{\partial U_n}{\partial\theta}(0)\geq\epsilon_0$
\end{lemma}
\begin{proof}
Let $M=\lbrace (z=x+iy, w=u+iv): v=\rho(z, \bar{z})\rbrace$ be a hypersurface of $\mathbb{C}^2$ and two open neighborhood $P$ and $Q$ of $q$ so that $(0, 0)\notin Q$ and $P\subsetneq Q\subsetneq\Sigma_+$. We first construct the images $D_n$ of the analytic discs in $\mathbb{C}$. Let $D_n$ be the simply connected bounded `eggs shaped' domain (See Figure \ref{fi}) with smooth boundaries with one vertex $q$ contained in $\Sigma_+$ such that $D_n\subset D_{n+1}$. We also, for convenience, let each of $D_n$ be symmetric with the angle bisector $\theta=\dfrac{\theta_{k+1}-\theta_k}{2}$ and all of $D_n\cap P$ coincides. We also assume the other vertex of $D_n$ approaches to $(0, 0)$ as $n$ goes to infinity.

\begin{figure}
   \centering
   \includegraphics{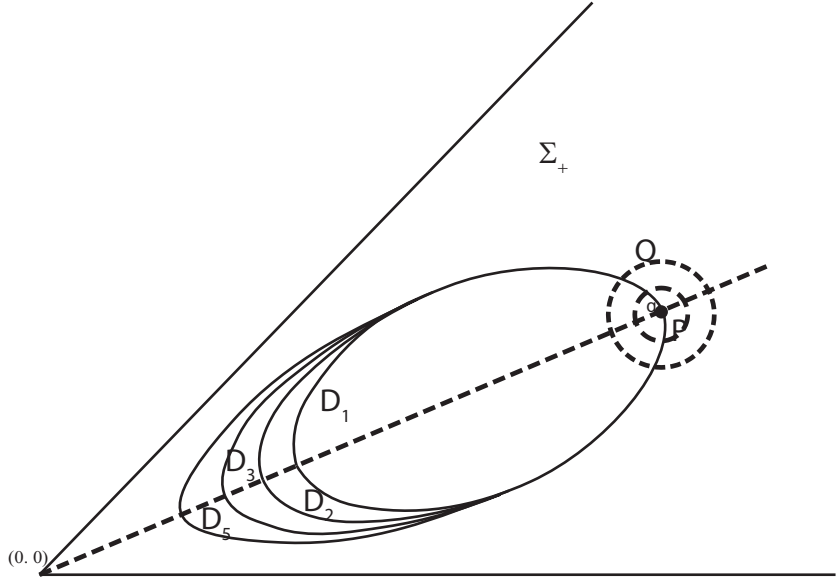}
   \caption{the `eggs shaped' domain.}
   \label{fi}
\end{figure}

Now, by Riemann map and probably composition with an M\"{o}bius trsnsform, one can find analytic discs $Z_n$ satisfying the following properties.
\begin{enumerate}
\item The images of $Z_n$ is exactly $D_n$.
\item $Z_n(1)$ is the other vertex other than $q$, hence $Z_n(1)\rightarrow (0, 0)$ as $n\rightarrow\infty$.
\item $Z_n(-1)=q$ for each $n$.
\end{enumerate}
By the Carath\'{e}odory theorem, all of $Z_n$ extends smoothly up to the boundaries. So $Z_n(e^{it})$ makes sense now (e.g. $Z_n(e^{i\pi})=Z_n(-1)=q$). For getting the lower bounds, we also let $Z_n$ satisfying the extra property.

\begin{enumerate}
\item[(4)] there exists $\dfrac{\pi}{2}<t_1<\pi$ and $-\pi<t_2<-\dfrac{\pi}{2}$ such that $Z_n(e^{i\frac{\pi}{2}}),Z_n(e^{-i\frac{\pi}{2}})\notin Q$ but $Z_n(e^{it})$ maps $[t_1, \pi)$ and $(-\pi, t_2]$ into $P$ for all of $n$.
\end{enumerate}

The property above is done again, by composition with m\"{o}bius transforms. Indeed the composition with $\dfrac{z-\alpha}{1-\bar{\alpha}z}$ adjusts the map $Z_n(e^{it})$ once $\alpha$ is real while it does not change the value at $\pi$ and $-\pi$. 

Let $A_n(z)=(Z_n(z)=X_n(z)+iY_n(z), W_n(z)=U_n(z)+iV_n(z))$ be an analytic disc attached to $M$ generated by $Z_n$ defined above. By the solution of the Bishop Equation $U_n(e^{i\theta})=T_1(\rho(Z_n(e^{i\theta}),\overline{Z_n}(e^{i\theta}))$, where $T_1$ is the modified Hilbert operator. 

Next, we will perturb the hypersurface $M$ a little around $q$ locally. This perturbation will be so small that it does not change the pseudocovexity. For this aim, we define the characteristic smooth function $\chi:\mathbb{C}\rightarrow\mathbb{R}$ such that
\begin{equation*}
  \chi(z,\bar{z})=%
  \begin{cases}
    1 &\text{if $z\in P$} \\
    0 &\text{if $z\notin Q$}.
  \end{cases}
\end{equation*}

Also, define $\tau=\chi\cdot\epsilon\vert z\vert^2$, where $\epsilon$ is a small positive real number so that $M':=\lbrace v=\rho+\tau\rbrace$ has the same pseudoconvexity as $M$. It is not hard to see $\tau$ in $P$ has lower bound, say $\tilde{\epsilon}$. Denote $\tau_n(e^{i\theta})=\tau(Z_n(e^{i\theta}))$.

Consider the hypersurface $v=\rho+\tau$ after perturbation, one observes the fact $\dfrac{\partial T_1\rho(Z_n(e^{i\theta}))}{\partial \theta}(0)+\dfrac{\partial T_1\tau_n(e^{i\theta})}{\partial \theta}(0)\geq 0$. The reason is same as the argument we used for $A'_n$ and the perturbation does not change the pseudoconvexity. Hence, we need to get a negative upper uniform bound of $\dfrac{\partial T_1\tau_n(e^{i\theta})}{\partial \theta}(0)$.

In fact, the modified Hilbert transform is given by,
\begin{equation*}
T_1\tau_n(e^{i\theta})=\mathrm{p.v.}\frac{1}{2\pi}\displaystyle\int_{-\pi}^\pi \!\frac{\tau_n(e^{i(\theta-t)})}{\tan(t/2)}\,\mathrm{d}t+C.
\end{equation*}

By changing variables, one obtains
\begin{equation*}
T_1\tau_n(e^{i\theta})=\mathrm{p.v.}\frac{1}{2\pi}\displaystyle\int_{-\pi+\theta}^{\pi+\theta} \!\frac{\tau_n(e^{it})}{\tan((\theta-t)/2)}\,\mathrm{d}t+C.
\end{equation*}

We claim that
\begin{equation}
\begin{split}
\mathrm{p.v.}\dfrac{1}{2\pi}\displaystyle\int_{-\pi+\theta}^{\pi+\theta} \!\dfrac{\tau_n(e^{it})}{\tan((\theta-t)/2)}\,\mathrm{d}t
=\displaystyle\int_{-\pi+\theta}^{-\delta_n}\!\dfrac{\tau_n(e^{it})}{\tan((\theta-t)/2)}\,\mathrm{d}t\\+
\displaystyle\int_{\delta'_n}^{\pi+\theta}\!\dfrac{\tau_n(e^{it})}{\tan((\theta-t)/2)}\,\mathrm{d}t 
\end{split}
\end{equation}
where $\delta_n, \delta'_n>0$. Indeed, this is because our $\tau$ just change $M$ around $A(-1)$ locally and most of values on unit circle vanishes.
 
Due to the claim, if $\vert\hat{\theta}_n\vert<\min\lbrace\dfrac{\delta_n}{2},\dfrac{\delta'_n}{2}\rbrace$ 
\begin{equation}\label{4p}
\begin{split}
\dfrac{\partial T_1\tau_n(e^{i\theta})}{\partial \theta}(\hat{\theta}_n)=-\dfrac{1}{2}\displaystyle\int_{-\pi+\hat{\theta}_n}^{-\delta_n}\!\dfrac{\tau_n(e^{it})}{\sin^2((\hat{\theta}_n-t)/2)}\,\mathrm{d}t-\dfrac{\tau_n(e^{i(-\pi+\hat{\theta}_n)})}{\tan(\pi/2)}\\-\dfrac{1}{2}\displaystyle\int_{\delta'_n}^{\pi+\hat{\theta}_n}\!\dfrac{\tau_n(e^{it})}{\sin^2((\hat{\theta}_n-t)/2)}\,\mathrm{d}t+\dfrac{\tau_n(e^{i(\pi+\hat{\theta}_n)})}{\tan(-\pi/2)}.
\end{split}
\end{equation}
In fact, $\vert\hat{\theta}_n\vert<\min\lbrace\dfrac{\delta_n}{2},\dfrac{\delta'_n}{2}\rbrace$ implies $\hat{\theta}_n+\delta_n>\delta_n-\dfrac{\delta_n}{2}=\dfrac{\delta_n}{2}>0$ for the first integral in (\ref{4p}). Hence, $0<\dfrac{\delta_n}{2}\leq\hat{\theta}_n-t\leq\pi$ and $\tan((\hat{\theta}_n-t)/2)$ are well defined (Riemann integrable). Similarly, one can show the other integral in (\ref{4p}) is also Riemann integrable. And (\ref{4p}) easily follows, in particular, 
\begin{equation}\label{im}
\begin{split}
\dfrac{\partial T_1\tau_n(e^{i\theta})}{\partial \theta}(0)=-\dfrac{1}{2}\displaystyle\int_{-\pi}^{-\delta_n}\!\dfrac{\tau_n(e^{it})}{\sin^2((-t)/2)}\,\mathrm{d}t-\dfrac{\tau_n(e^{-i\pi})}{\tan(\pi/2)}\\-\dfrac{1}{2}\displaystyle\int_{\delta'_n}^{\pi}\!\dfrac{\tau_n(e^{it})}{\sin^2((-t)/2)}\,\mathrm{d}t+\dfrac{\tau_n(e^{i\pi})}{\tan(-\pi/2)}.
\end{split}
\end{equation}
Investigating $\displaystyle\int_{-\pi}^{-\delta_n}\!\dfrac{\tau_n(e^{it})}{\sin^2((-t)/2)}\,\mathrm{d}t\geq\displaystyle\int_{-\pi}^{t_2}\!\dfrac{\tau_n(e^{it})}{\sin^2((-t)/2)}\,\mathrm{d}t\geq\displaystyle\int_{-\pi}^{t_2}\!\tilde{\epsilon}\,\mathrm{d}t=(t_2+\pi)\tilde{\epsilon}$ and similarly $\displaystyle\int_{\delta'_n}^{\pi}\!\dfrac{\tau_n(e^{it})}{\sin^2((-t)/2)}\,\mathrm{d}t\geq (\pi-t_1)\tilde{\epsilon}$ while the second of the forth terms in equation \ref{im} are zeros. Hence, the negative upper uniform bound of $\dfrac{\partial T_1\tau_n(e^{i\theta})}{\partial \theta}(0)$ is $-\dfrac{1}{2}(t_2-t_1+2\pi)\tilde{\epsilon}$. We denote the upper bound $-\epsilon_0<0$.

Thus, we have $\dfrac{\partial U_n}{\partial\theta}(0)=\dfrac{\partial T_1\rho(Z_n(e^{i\theta}))}{\partial \theta}(0)\geq -\dfrac{\partial T_1\tau_n(e^{i\theta})}{\partial \theta}(0)\geq\epsilon_0$.
\end{proof}

So far, we exhibit a family of analytic discs $A_n$ with their generator $Z_n$. The nice property of $A_n$ is $\dfrac{\partial U_n}{\partial\theta}(0)\geq\epsilon_0>0$ . We are going to treat $Z_n$ as a family of candidates and try to translate them so that after it, $\partial D_n$ passes through $(0, 0)$. In other words, we will look at $Z_n+c_n$ where $c_n=(a_n, b_n)\in\mathbb{C}$ are constant vector determined by $Z_n$ so that $Z_n(1)+c_n=(0, 0)$. Since our move is translation and $Z_n(1)$ approaches to $(0, 0)$, also because of the fact $\rho$ is a smooth function  we can expect $\dfrac{\partial T_1(\rho(Z_n+c_n))}{\partial\theta}(0)$ will not change much from $\dfrac{\partial T_1(\rho(Z_n))}{\partial\theta}(0)$. We also hope the difference is smaller as $\dfrac{\epsilon_0}{2}$ for some $Z_{n_0}$ and then we can conclude $\dfrac{\partial T_1(\rho(Z_{n_0}+c_{n_0}))}{\partial\theta}(0)\geq\epsilon_0-\dfrac{\epsilon_0}{2}=\dfrac{\epsilon_0}{2}>0$. And then we will use the analytic disc $A$ generated from $Z_{n_0}+c_{n_0}$ to complete the final argument of holomorphic extension at $(0, 0)$. In fact, we have the following lemma.

\begin{lemma}
There exits $n_0$ such that $\left\vert\dfrac{\partial T_1(\rho(Z_{n_0}+c_{n_0}))}{\partial\theta}(0)-\dfrac{\partial T_1(\rho(Z_{n_0}))}{\partial\theta}(0)\right\vert\leq\dfrac{\epsilon_0}{2}$.
\end{lemma}
\begin{proof}
We see $\rho$ as a $C^{1,0.9}$ function of $\mathbb{C}$, although it is in fact a smooth function. Then 

\begin{equation*}
\begin{split}
\left\vert\dfrac{\partial T_1(\rho(Z_n+c_n))}{\partial\theta}(0)-\dfrac{\partial T_1(\rho(Z_n))}{\partial\theta}(0)\right\vert&=\left\vert\dfrac{\partial T(\rho(Z_n+c_n))}{\partial\theta}(0)-\dfrac{\partial T(\rho(Z_n))}{\partial\theta}(0)\right\vert\\&\leq\left\| T(\rho(Z_n+c_n))-T(\rho(Z_n))\right\|_{1, 0.9}
\end{split}
\end{equation*}

where $T$ is the (classical) Hilbert transform and $\parallel\cdot\parallel_{1, 0.9}$ is the norm of H\"{o}lder space $C^{1, 0.9}$. The first equality is because that $T_1$ is $T$ up to a constant and it vanishes after derivative. Since the Hilbert transform is a bounded linear operator from $C^{k,\alpha}(S^1)$ into itself when $k$ is a nonnegative integer and $0<\alpha<1$,

\begin{equation*}
\begin{split}
\left\| T(\rho(Z_n+c_n))-T(\rho(Z_n))\right\|_{1, 0.9}&\leq C_{1, 0.9}\left\|\rho(Z_n+c_n)-\rho(Z_n)\right\|_{1,0.9}\\&\leq C_{1, 0.9} K\left\|Z_n+c_n-Z_n\right\|_{1,0.9}
\end{split}
\end{equation*}
for some positive $K$, because of the smoothness of $\rho$. However, by observing the changes of the translation, $\left\|(Z_n+c_n)-Z_n\right\|_{1,0.9}=\mid c_n\mid$ because the translation does not change derivatives. Since $\mid c_n\mid\rightarrow 0$, we can pick up an $n_0$ such that $\mid c_{n_0}\mid<\dfrac{\epsilon_0}{2KC_{1, 0.9}}$ and this $n_0$ is what we need which competes the proof.
\end{proof}

\begin{remark*}
Note that the image of $Z_n+c_n$ might not be contained in $\Sigma_+\cup (0, 0)$ anymore.
\end{remark*}

\begin{proof}[Proof of Proposition \ref{prop}]
Let $A$ be a smooth analytic disc attached to $M$ generated from $Z_{n_0}+c_{n_0}$, of which the exit vector $\dfrac{\partial A}{\partial r}(1)$ points to the pseudoconvex side of $M$ at $(0, 0)$. By investigating  $\Sigma_-\cup(0, 0)$, it is not hard to see the existence of another smooth analytic disc $B$ attached to $M$ with the exit vector pointing pseudoconvex side. Note that pseudoconvex sides in this paragraph are, in fact, in different side, because $P_z^{-1}\Sigma_+\cap M$ is pseudoconvex but the other is pseudoconcave in sense of Definition \ref{pseudoconvexity}.

By again, the argument of the translation of a nontangent analytic disc (see 2.12 of Chapter V in \cite{MP06}), the proposition follows.
\end{proof}

\section{Corollaries}

Theorem \ref{mainth} has several applications. The main point of it is to remove ``pseudoconvex'' condition from known results, e.g. in \cite{By03} and \cite{BP88}. We do not want to write out all of them because the wish for the conciseness of this note.

However, we mention that, one of most important applications is the one obtained by combining result of Theorem \ref{mainth} with that of Bedford-Pinchuk's theorem in \cite{BP91}. 

\begin{corollary}
Let $\Omega\subset\mathbb{C}^2$ be a bounded domain with smooth boundary of finite type. Suppose that the Bergman kernel of $\Omega$ extends to $\overline{\Omega}\times\overline{\Omega}$ minus the boundary diagonal set as a locally bounded function. There is a hyperbolic orbit boundary accumulation point on $\partial\Omega$ if and only if $\Omega$ is biholomorphic to one of the ellipsoids $\lbrace (z, w): \vert z\vert^{2m}+\vert w\vert^2<1, m\in\mathbb{Z}^+\rbrace$. \qed
\end{corollary}

\bigskip
\bigskip
\noindent {\bf Acknowledgments}. I thank a lot to my advisor Prof. Steven Krantz for always being patience to answer my any (even stupid) question who also encouraged me to write up this paper and carefully read the article. I also thank to Prof. John McCarthy for supporting the author's research financially. Moreover, I have much profited from the comments of the (anonymous) reviewer and a special thanks go to him/her. Last but not least, I appreciate Prof. \'{A}lvaro Pelayo for teaching me Symplectic Geometry where I studied the idea of deformation.


\bibliographystyle{plainnat}
\bibliography{mybibliography}
\end{document}